\documentclass[12pt]{amsart}
\usepackage{amsmath,amsthm,amssymb,fullpage,caption}
\usepackage{tikz}
\usetikzlibrary{calc}
\usetikzlibrary{shapes.symbols}
\usetikzlibrary{arrows}
\usetikzlibrary{decorations.markings}
\usetikzlibrary{shapes}
\newcommand{\size}[1]{\left \vert #1 \right \vert}

\newcommand{\class}[1]{#1}

\newtheorem{theorem}{Theorem}[section]
\newtheorem{lemma}[theorem]{Lemma}
\newtheorem{conj}[theorem]{Conjecture}

\begin{document}

\title{Cops and Robbers is \class{EXPTIME}-complete}

\author{William B. Kinnersley}
\address{Department of Mathematics, University of Rhode Island, University of Rhode Island, Kingston, RI,USA, 02881}
\email{\tt billk@mail.uri.edu}


\begin{abstract}
We investigate the computational complexity of deciding whether $k$ cops can capture a robber on a graph $G$.  Goldstein and Reingold (\cite{GR95}, 1995) conjectured that the problem is \class{EXPTIME}-complete when both $G$ and $k$ are part of the input; we prove this conjecture.
\end{abstract}

\maketitle

\begin{section}{Introduction}

{\em Cops and Robbers} is a vertex-pursuit game that has received much recent attention.  The game has two players, a set of $k$ cops and a single robber.  The cops and robber occupy vertices of a graph $G$; more than one entity may occupy a single vertex.  The game has perfect information, meaning that at all times, both players have full knowledge of the graph and of all moves played thus far, and this knowledge determines the players' strategies; in particular, both players know the positions of the cops and the robber.  At the beginning of the game, the players freely choose vertices to occupy, with the cops choosing first and the robber last.  The game proceeds in {\em rounds}, each consisting of a cop turn followed by a robber turn.  In each round, each cop may remain on her current vertex or move to an adjacent vertex, after which the robber likewise chooses to remain in place or move to an adjacent vertex.  The cops win if any cop ever occupies the same vertex as the robber; we call this {\em capturing} the robber.  Conversely, the robber wins by evading capture forever.

Placing one cop on each vertex of $G$ ensures an immediate win for the cops.  We are thus justified in defining the {\em cop number} of $G$, denoted $c(G)$, to be the minimum number of cops needed to capture a robber on $G$ (regardless of how the robber plays).  Determination of the cop number is the central problem in the study of Cops and Robbers.  Quilliot~\cite{Qui78}, along with Nowakowski and Winkler~\cite{NW83}, independently introduced the game and characterized graphs with cop number 1.  Aigner and Fromme~\cite{AF84} initiated the study of graphs having larger cop numbers.  For more background on Cops and Robbers, we direct the reader to~\cite{BN11}.

In this paper, we investigate the computational complexity of deciding whether a graph has cop number at most $k$; this is a natural question, as the game of Cops and Robbers has applications in the field of Artificial Intelligence (see for example~\cite{ILBG08,MS09}).  More precisely, we study the associated decision problem
$$\mbox{\textsc{C\&R}: Given a graph $G$ and positive integer $k$, is $c(G) \le k$?}$$
\textsc{C\&R} clearly belongs to \class{EXPTIME} (the class of decision problems solvable in exponential time), since the number of possible game states is bounded above by $n^{k+1}$, where $n = \size{V(G)}$.  (This observation also implies that the problem is solvable in polynomial time when $k$ is fixed in advance, rather than being considered part of the input; see also~\cite{BC09,BI93,CM12,HM06}.)  Goldstein and Reingold~\cite{GR95} showed that the generalization of \textsc{C\&R} in which $G$ may be directed is \class{EXPTIME}-complete.  They also proved \class{EXPTIME}-completeness when the initial positions of the cops and robber are given as part of the input.  More recently, Fomin, Golovach, and Pra\l{}at~\cite{FGP12} showed that the problem is \class{PSPACE}-complete under various restrictions on the duration of the game.

Goldstein and Reingold were unable to establish any hardness results regarding the original game, in which $G$ is undirected and the players may choose their initial positions.  They did, however, pose the following conjecture:
\begin{conj}[Goldstein and Reingold, 1995]\label{conj:main}
\textsc{C\&R} is \class{EXPTIME}-complete. 
\end{conj}
Since its introduction, this conjecture has been one of the foremost open questions in the game of Cops and Robbers.  It has proved quite resistant to attack; it was not until 2010 that the problem was even known to be NP-hard~\cite{FGKNS10}!  More recently, Mamino~\cite{Mam13} proved \class{PSPACE}-hardness through reduction to a new variant, ``Cops and Robbers with protection'', which in turn reduces to the original game.

In this paper, we prove Conjecture~\ref{conj:main}.  Our approach is akin to Mamino's in that we introduce a new variant of the game that further hampers the cops.  We first show that this new variant reduces to Mamino's ``Cops and Robbers with protection'', which in turn reduces to \textsc{C\&R}.  We then prove \class{EXPTIME}-completeness of our new variant through reduction from a known \class{EXPTIME}-complete problem.

The paper is structured as follows.  In Section~\ref{sec:prelim}, we introduce the game of ``Lazy Cops and Robbers'' and show that it reduces to \textsc{C\&R}; this new game serves as an intermediate step in our overall reduction.  Section~\ref{sec:main} contains the proof of Conjecture~\ref{conj:main}.  Since the final reduction is rather complicated, we present and explain the construction in Section~\ref{sec:construction}, relegating the proof of correctness to Section~\ref{sec:proof}.

Throughout the paper we consider only finite, undirected graphs.  We denote the vertex set and edge set of a graph $G$ by $V(G)$ and $E(G)$, respectively.  We use $N(v)$ for the neighbourhood of a vertex $v$; to emphasize the graph $G$ under consideration, we sometimes write $N_G(v)$.  We use $L(G)$ to denote the line graph of $G$, that is, the graph having one vertex for each edge of $G$, with two vertices adjacent when the corresponding edges in $G$ share a common endpoint.  We write $[n]$ as shorthand for $\{1, 2, \ldots, n\}$.  For more notation and background in graph theory, we direct the reader to~\cite{Wes01}.
\end{section}

\begin{section}{Preliminaries}\label{sec:prelim}

In this section, we introduce the game of ``Lazy Cops and Robbers'' and show that its decision problem reduces to \textsc{C\&R}.  Thus in Section~\ref{sec:main} we may work with this new game instead of the original, which greatly simplifies the arguments.  

As an intermediate step, we make use of the game of ``Cops and Robbers with protection'' introduced by Mamino~\cite{Mam13}.  This variant differs from ordinary Cops and Robbers in two ways.  First, each edge of $G$ is designated (in advance, as part of the input) either {\em protected} or {\em unprotected}.  Second, the cops win in the variant only if a cop comes to occupy the same vertex as the robber by traveling along an unprotected edge.  (The cops may still travel across protected edges, but cannot capture across them.)  In particular, the robber may move to a vertex containing a cop without immediately losing the game; in fact, he may even choose to start the game on the same vertex as a cop.  We denote the decision problem associated with this game by \textsc{C\&Rp}.

We adopt the convention that all input graphs for \textsc{C\&Rp} are reflexive (that is, every vertex has a loop).  We may do this without loss of generality, since a vertex without a loop is functionally equivalent to a vertex with a protected loop.  We refer to a vertex with a protected loop as a {\em protected vertex} and to a vertex with an unprotected loop as an {\em unprotected vertex}.  Note that if the robber moves to an unprotected vertex on which a cop is already present, then that cop may capture the robber on her next turn by traversing the loop at that vertex.

Mamino showed that it suffices to consider \textsc{C\&Rp} in place of \textsc{C\&R}.  More precisely, he proved the following:

\begin{lemma}[\cite{Mam13}, Lemma 3.1]\label{lem:mamino}
\textsc{C\&Rp} is \class{LOGSPACE}-reducible to \textsc{C\&R}.
\end{lemma}

To prove Lemma~\ref{lem:mamino}, Mamino provides an construction that, given an instance $(G,k)$ of \textsc{C\&Rp}, produces an instance $(G',k)$ of \textsc{C\&R} that is ``equivalent'' in the sense that the cops have a winning strategy for the \textsc{C\&Rp} game on $G$ if and only if they also have a winning strategy for the \textsc{C\&R} game on $G'$.  Loosely, the graph $G'$ is produced by making many copies of each vertex of $G$, then adding to $G'$ all edges that correspond to unprotected edges in $G$, along with some (but not all) edges corresponding to protected edges.  This latter class of edges must be chosen very carefully; the intent is to provide the mobility offered by the protected edges of $G$, without introducing any ``new'' winning strategies for the cops.  To accomplish this, Mamino adds edges corresponding to some preselected graph $H$ on which the robber can always evade the cops; thus the robber can always navigate these new edges in such a way that cops traveling along such edges cannot capture him.

We aim to show that the decision problem for Lazy Cops and Robbers reduces to \textsc{C\&Rp}.  In fact we use a more general variant, {\em Lazy Cops and Robbers with protection}, whose decision problem we denote by \textsc{LC\&Rp}.  The rules of \textsc{LC\&Rp} are the same as in \textsc{C\&Rp}, except that on each turn, at most one cop may move to a new vertex.  (The cops may collectively decide who moves.)  This change makes \textsc{LC\&Rp} substantially easier to work with than \text{C\&Rp}.  Thus we would like to focus on \textsc{LC\&Rp} instead of \textsc{C\&Rp}; our next result shows that we can do just that.

Throughout the paper, we say that a cop {\em defends} a vertex $v$ when she occupies an adjacent vertex $u$ such that $uv$ is unprotected; that is, she could capture a robber at $v$ on her next turn.  We say that the robber can {\em safely} move to an adjacent vertex $v$ when he may move there without immediately being captured, that is, when $v$ is undefended.  When a robber has only one move that prevents capture on the next turn, we say that he {\em must} make that move; likewise, when a move would result in immediate capture, we say that he {\em must not} make it.  Similarly, we say that the cops {\em must not} make any move that clearly does not benefit them, and {\em must} make a move when no others provide any benefit.  (We make this notion more precise where it is used.)

\begin{theorem}\label{thm:lcrdp}
\textsc{LC\&Rp} reduces to \textsc{C\&Rp} in polynomial time.
\end{theorem}
\begin{proof}
Given an instance $(G,k)$ of \textsc{LC\&Rp}, we produce a graph $G'$ such that $k$ cops can capture a robber on $G$ in the \textsc{LC\&Rp} game if and only if $k$ cops can capture a robber on $G'$ in the \textsc{C\&Rp} game.  We may assume $k \ge 2$, since otherwise the reduction is trivial (simply take $G' = G$).

All vertices and edges of $G'$ are unprotected except where otherwise specified.  We begin with three disjoint subgraphs: $G_C$, $G_R$, and $L_R$.  The graph $G_R$ is isomorphic to $G$, and $L_R$ is isomorphic to $L(G)$.  To form $G_C$, we begin with $G$ and subdivide each edge once; we refer to the original vertices of $G$ as {\em branch vertices} and the vertices arising from subdivision as {\em subdivision vertices}.  (We adopt the convention that when a loop is subdivided, the subdivision vertex becomes a leaf adjacent to the relevant branch vertex; that is, we avoid producing a double-edge.)  Next we replace each branch vertex $v$ with a copy of $K_k$, each vertex of which we make adjacent to all neighbours of $v$; we refer to the copy of $K_k$ as the {\em branch clique} for $v$.  All vertices and edges within $G_C$, $G_R$, and $L_R$ are unprotected.

For a vertex $v$ and edge $e$ in $G$, we denote the branch clique for $v$ by $\lambda_C(v)$, the subdivision vertex corresponding to $e$ by $\lambda'_C(e)$, the vertex in $G_R$ corresponding to $v$ by $\lambda_R(v)$, and the vertex of $L_R$ corresponding to $e$ by $\lambda'_R(e)$.  We index the vertices in each branch clique from $1$ to $k$ arbitrarily; we refer to those vertices with index 1 as {\em 1-vertices}, those with index 2 as {\em 2-vertices}, and so on.  For reasons that will become clear later, we additionally consider subdivision vertices to be 1-vertices.

We now add edges joining the subgraphs.  For each edge $uv$ in $G$ (including loops), we add edges from $\lambda'_R(uv)$ to both $\lambda_R(u)$ and $\lambda_R(v)$.  Additionally, we add edges from every vertex in $\lambda_C(u)$ to $\lambda_R(v)$ and from every vertex in $\lambda_C(v)$ to $\lambda_R(u)$; these edges are protected if and only if $uv$ is protected in $G$.  We also add edges joining all subdivision vertices to all vertices in $G_R$, as well as edges joining all 1-vertices in branch cliques to all vertices of $L_R$.

Next we add a gadget to help simulate the constraint that only one cop may move at a time and, additionally, to enforce the desired initial positions.  We add new vertices $r_1, r_2, \ldots, r_k$, all protected.  We also add all edges of the form $r_ir_j$, all protected.  Next, for each $i \in [k]$, we add (unprotected) edges joining $r_i$ to all $i$-vertices in $G_C$.  Finally, we add protected edges joining each $r_i$ to all vertices in $G_R$ and $L_R$.  We refer to $\{r_1, r_2, \ldots, r_k\}$ as the {\em reset clique}.  As we show later, should the cops ``misbehave'', the robber can use the reset clique to return the game to a ``proper'' configuration.  Likewise, should the robber misbehave, the cops can use the reset clique to capture him.

For technical reasons, we add one final vertex $\omega$ and edges from $\omega$ to all vertices in $G_C$, $G_R$, and $L_R$.  Figure~\ref{fig:lcrdp} shows an example of the construction for $k=2$.  The original graph $G$ appears on the left, and the new graph $G'$ appears on the right.  In the figure, black vertices are unprotected, white vertices are protected, solid edges are unprotected, and dashed edges are protected.  

\begin{figure}[h]
\begin{center}
\includegraphics{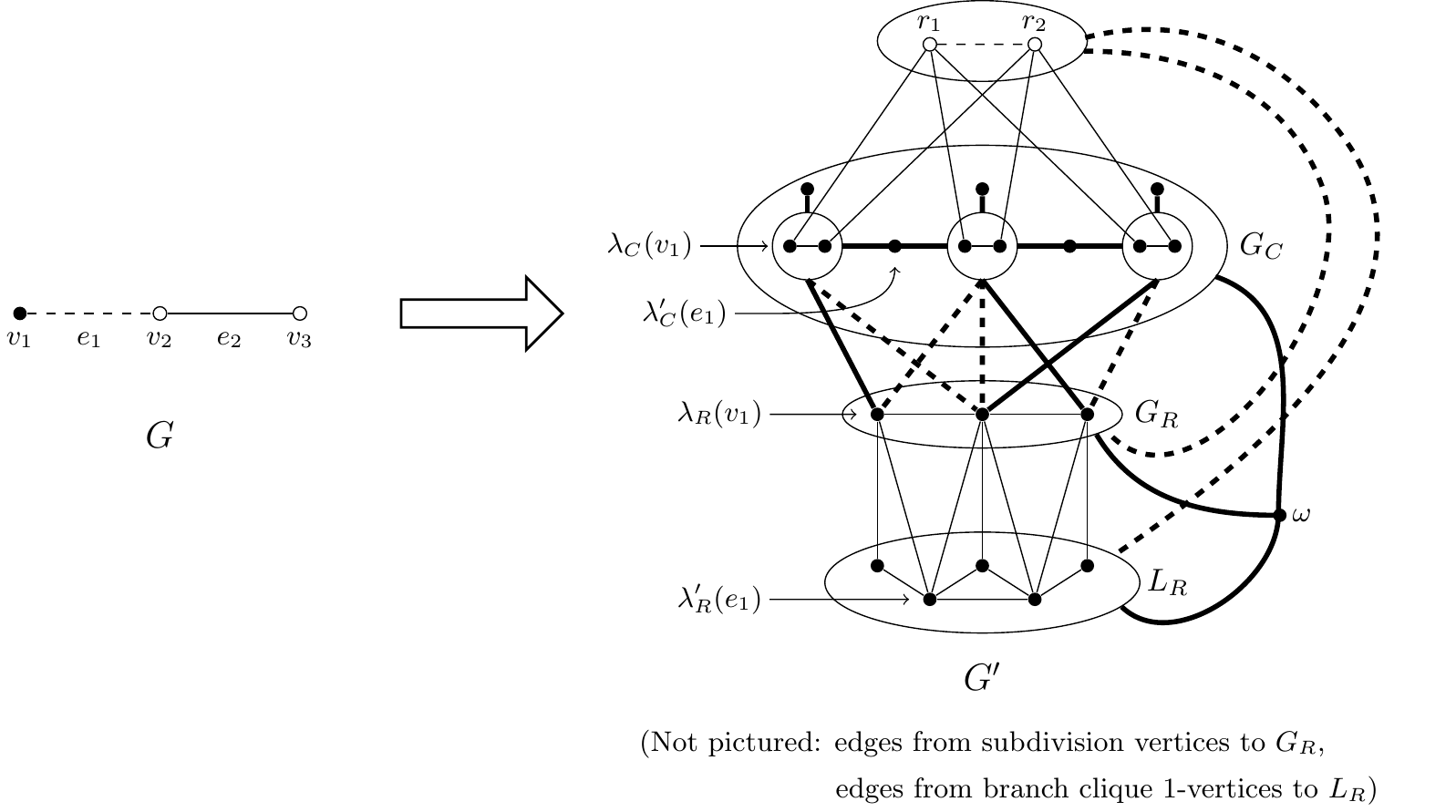}
\end{center}
\caption{. The construction in Theorem~\ref{thm:lcrdp} (for $k=2$).}\label{fig:lcrdp}
\end{figure}

We claim that $k$ cops can capture a robber on $G$ in the \textsc{LC\&Rp} game on $G$ if and only if $k$ cops can capture a robber on $G'$ in the \textsc{C\&Rp} game.  Before presenting specific cop and robber strategies, we make some useful observations about the game on $G'$.  Throughout, we refer to vertices in $G_C$ as {\em cop vertices} and to those in $G_R$ and $L_R$ as {\em robber vertices}. Let a {\em basic configuration} be any configuration of the game on $G'$ in which the robber occupies a robber vertex and all cops occupy cop vertices, with exactly one cop on an $i$-vertex for all $i \in [k]$.


\medskip\medskip
\noindent {\bf Claim 1}: If there is a cop on an $i$-vertex for all $i \in [k]$, it is the cops' turn, and the robber occupies a cop vertex, then the cops can capture the robber.\\

Suppose the robber occupies some cop vertex $v$, and let $v$ be a $j$-vertex.  To capture the robber, the cop currently on a $j$-vertex remains in place, while some other cop moves to $\omega$.  On the robber's subsequent turn, he may only reach cop vertices, robber vertices, $\omega$, and $r_j$.  However, the cop on a $j$-vertex defends $r_j$, while the cop on $\omega$ defends $\omega$ itself, all cop vertices, and all robber vertices.

\medskip\medskip
\noindent {\bf Claim 2}: If it is the robber's turn, the robber occupies a robber vertex, and the game is not in a basic configuration, then the robber can safely move to the reset clique.\\

For some $j$, no cop occupies a $j$-vertex.  Hence $r_j$ is undefended, so the robber can safely move there.

\medskip\medskip
\noindent {\bf Claim 3}: If the robber stands in the reset clique and it is the robber's turn, then the robber can force the game to (eventually) return to a basic configuration in which it is the cops' turn and the robber occupies any robber vertex he chooses.\\

If some vertex in the reset clique is undefended, then the robber can move (or remain) there, repeating until the cops defend all vertices in the reset clique.  This requires a cop on an $i$-vertex for all $i \in [k]$, so at this point the robber may produce a basic configuration by moving to any robber vertex.

\medskip\medskip
\noindent {\bf Claim 4}: At the beginning of the game, either the cops or the robber may force the game to (eventually) reach a basic configuration in which it is the cops' turn and the robber occupies any robber vertex he chooses.\\

The cops need simply place one cop on an $i$-vertex for each $i \in [k]$.  This defends all vertices in the reset clique, along with $\omega$.  By Claim 1 the robber must not start on a cop vertex, so he must start in either $G_R$ or $L_R$, resulting in a basic configuration.

If the cops do not start in such a configuration, then some vertex in the reset clique is undefended, and the robber can safely start there.  From there, he can force the game into a basic configuration by Claim 3.

\medskip\medskip
\item {\bf Claim 5}: From a basic configuration, the cops and robber must not move to $\omega$.\\

By Claim 3, if the robber ever safely reaches the reset clique, then he can force the game into a basic configuration; this does not help the cops, since by Claim 4) they could have started the game in this configuration to begin with.  Hence the cops must not make any moves that allow the robber safe access to the reset clique.  With a cop on $\omega$ the game is no longer in a basic configuration, and by Claim 2 the robber can move to the reset clique.  Should the robber move to $\omega$ from a basic configuration, any cop can capture him.\\

\medskip\medskip
By Claim 4, either player can initially force the game into a basic configuration.  By Claims 1, 2, 3, and 5, along with the observation that in a basic configuration the cops defend all vertices of the reset clique, it follows that the game must remain in a basic configuration until some cop can capture the robber (and thus end the game) on her next move.  We are now ready to present the cop and robber strategies.\\

Suppose $k$ cops can win the \textsc{LC\&Rp} game on $G$; we show how $k$ cops can win the \textsc{C\&Rp} game on $G'$.  In addition to the \textsc{C\&Rp} game on $G'$, the cops ``imagine'' an instance of the \textsc{LC\&Rp} game on $G$ and play both games simultaneously, using a winning strategy for the \textsc{LC\&Rp} game to guide their play in the \textsc{C\&Rp} game.

Label the cops in the \textsc{LC\&Rp} game $C_1, C_2, \ldots, C_k$, and label those in the \textsc{C\&Rp} game $C'_1, C'_2, \ldots, C'_k$.  For each $i \in [k]$, let $v_i$ denote the starting vertex of cop $C_i$ in the \textsc{LC\&Rp} game; in the \textsc{C\&Rp} game, cop $C'_i$ begins on the $i$-vertex in $\lambda_C(v_i)$.  As argued in Claim 4, the robber must begin on a robber vertex.  Since cop $C'_1$ defends all vertices in $L_R$, the robber must begin on $\lambda_R(v)$ for some $v \in V(G)$.  The cops now pretend that the robber in the \textsc{LC\&Rp} game started on $v$. 

While playing the games, the cops maintain the following invariant: if the cops in the \textsc{LC\&Rp} game occupy $w_1, w_2, \ldots, w_k$ and the robber occupies $w_R$, then in the \textsc{C\&Rp} remains in a basic configuration, with one cop occupying a vertex of each $\lambda_C(w_i)$ and the robber occupying $\lambda_R(w_R)$.  We show how the cops can maintain this invariant while mimicking, in the \textsc{C\&Rp} game, their winning strategy for the \textsc{LC\&Rp} game.  

Suppose the cop strategy in the \textsc{LC\&Rp} game calls for cop $C$ to move from vertex $w$ to vertex $x$.  In the \textsc{C\&Rp} game, this requires two steps.  By the invariant, some cop $C'$ currently sits on some $j$-vertex $w'$ in $\lambda_C(w)$.  The cop $C'$ moves to the subdivision vertex corresponding to edge $wx$; if $j \not = 1$, then additionally the cop currently sitting on a 1-vertex moves to the $j$-vertex in her branch clique.  The cops now defend all of $G_R$.  By Claims 1 and 5, the robber must move to $L_R$; say he moves from vertex $\lambda_R(y)$ in $G_R$ to $\lambda'_R(yz)$ in $L_R$.  Cop $C'$ now moves to the 1-vertex in $\lambda_C(x)$, which defends all vertices in $L_R$ and forces the robber to move either to $\lambda_R(y)$ or $\lambda_R(z)$.  The cops interpret these actions as moves in the \textsc{LC\&Rp} game: the cops move a cop from $w$ to $x$, while the robber either remains on $y$ or moves to $z$, as appropriate.  (Note that we may have $w = x$, $y = z$, or both.)  In either case, the invariant is maintained.

Eventually the \textsc{LC\&Rp} game reaches a configuration in which the cops can capture the robber, say by moving a cop from $w$ to $x$.  By the invariant, some cop $C$ sits in  $\lambda_C(w)$, while the robber occupies $\lambda_R(x)$.  Now $C$ moves to $\lambda_R(x)$; since $wx$ is necessarily unprotected in $G$, all edges from $\lambda_C(w)$ to $\lambda_R(x)$ are unprotected in $G'$, hence the cops win.\\

Suppose now that the robber can win the \textsc{LC\&Rp} game on $G$; we show how he can win the \textsc{C\&Rp} game on $G'$.  As above, the robber imagines an instance of the \textsc{LC\&Rp} game on $G$, plays both games simultaneously, and uses a winning strategy for the \textsc{LC\&Rp} game to win the \textsc{C\&Rp} game.  By Claim 4, we may suppose without loss of generality that the game begins in a basic configuration.  There are two possibilities: either the cop on a 1-vertex starts in a branch clique, or she occupies subdivision vertex.  We first suppose the former, and afterward explain how the robber may deal with the latter.

Label the cops $C_1, C_2, \ldots, C_k$.  Suppose cop $C_i$ occupies vertex $w_i$, where $w_i \in \lambda_C(v_i)$.  The robber proceeds as if the cops in the \textsc{LC\&Rp} game began on $v_1, v_2, \ldots, v_k$.  Let $v$ be his starting position in that game; in the \textsc{C\&Rp} game, he starts on $\lambda_R(v)$.  As before we maintain the invariant that if the cops in the \textsc{LC\&Rp} game occupy $w_1, w_2, \ldots, w_k$ and the robber occupies $w_R$, then one cop in the \textsc{C\&Rp} game occupies a vertex of each $\lambda_C(w_i)$, and the robber occupies $\lambda_R(w_R)$.  As a consequence of the invariant and the construction on $G'$, if the robber is not currently threatened in the \textsc{LC\&Rp} game, then also he is not threatened in the \textsc{C\&Rp} game.  

Suppose first that no cops occupy subdivision vertices.  Maintaining a basic configuration requires at most one cop on a subdivision vertex, so in each round at most one cop can leave her branch clique.  If no cops leave their cliques, then the robber remains still.  Suppose instead that cop $C$, currently on some vertex of $\lambda_C(w)$, leaves her branch clique, while all other cops remain in theirs.  By Claim 2, $C$ either captures the robber or moves to a subdivision vertex corresponding to some edge $wx$.  Since the robber is not threatened in the \textsc{LC\&Rp} game the former is impossible, so we may suppose the latter.  The robber interprets this move as $C$ signaling her intent to move to $x$ in the \textsc{LC\&Rp} game.  Let the robber's current vertex be $\lambda_R(y)$.  In the \textsc{LC\&Rp} game, the robber sits on $y$; say his strategy tells him that, should the cop on $w$ move to $x$, the robber should respond by moving to $z$.  In the \textsc{C\&Rp} game, the robber now moves ``halfway'' to $\lambda_R(z)$ by moving to $\lambda'_R(yz)$.  (As with the cops' strategy, we could have $w=x$, $y=z$, or both.)  There are four possibilities.
\begin{enumerate}
\item {\em If $C$ does not move on the next turn}, then no other cops may leave their branch cliques, the robber sits still, and the \textsc{LC\&Rp} game remains unchanged.\\
\item {\em If $C$ moves back to $\lambda_C(w)$ and no other cop leaves her branch clique}, then the robber moves back to $\lambda_R(y)$, and the \textsc{LC\&Rp} game remains unchanged.\\
\item {\em If $C$ moves to $\lambda_C(x)$ and no other cop leaves her branch clique}, then the robber moves to $\lambda_R(z)$ and advances the \textsc{LC\&Rp} game by moving a cop from $w$ to $x$ and, in response, moving the robber from $y$ to $z$.\\
\item {\em If $C$ moves to $\lambda_C(w)$ or $\lambda_C(x)$ and some other cop $C'$ leaves her branch clique}, then the robber treats these as two separate moves in the following sense.  First, he pretends that $C$ moves while $C'$ remains still, decides where to move in response, and advances the \textsc{LC\&Rp} game accordingly.  After that, he pretends that $C'$ makes her move, and decides which vertex $\lambda'_R(e)$ of $L_R$ he would move to in response.  The robber then moves directly to $\lambda'_R(e)$; this is possible because $yz$ and $e$ must share an endpoint in $G$, hence the vertices $\lambda'_R(yz)$ and $\lambda'_R(e)$ are adjacent in $L_R$.
\end{enumerate}

As argued above, the cops cannot deviate from this formula unless some cop can capture the robber by moving to $G_R$.  However, by the invariant, the cops threaten to capture the robber in the \textsc{C\&Rp} game if and only if they threaten to capture the robber in the \textsc{LC\&Rp} game.  Since the robber avoids capture indefinitely in the \textsc{LC\&Rp} game, he thus avoids capture indefinitely in the \textsc{C\&Rp} game.

Finally, suppose one cop starts on the subdivision vertex corresponding to some edge $wx$.  The robber pretends that this cop had started on $w$ and, subsequently, moved from $\lambda_C(w)$ to $\lambda'_C(wx)$.  The robber chooses a starting position and first move based on this hypothetical scenario, takes note of the vertex $v$ at which he winds up, and starts directly on $v$ in the ``real'' \textsc{C\&Rp} game.\\

It is clear that $G'$ may be constructed in polynomial time, which completes the proof.
\end{proof}  
\end{section}

\begin{section}{Proof of the main result}\label{sec:main}

We are now ready to prove Conjecture~\ref{conj:main}.  To do so, it suffices to reduce a known \class{EXPTIME}-complete problem to \textsc{LC\&Rp}; Conjecture~\ref{conj:main} then follows by Lemma~\ref{lem:mamino} and Theorem~\ref{thm:lcrdp}.  We use the problem known as the {\em Alternating Boolean Formula} game (or {\em \textsc{ABF}}), shown to be \class{EXPTIME}-complete by Stockmeyer and Chandra~\cite{SC79}\footnote{In fact, Stockmeyer and Chandra showed that ABF is complete for the complexity class ETIME, the class of problems solvable in time $2^{O(n)}$.  However, it is well-known that every ETIME-complete problem is also EXPTIME-complete; see for example~\cite{Hom97}, Theorem 1.1.}.  In \textsc{ABF}, players A and B are given disjoint sets $X$ and $Y$ of variables, with prescribed initial values, along with a boolean formula $\varphi$ in conjunctive normal form.  The players play alternately, each changing the values of at most one of their variables.  Player A wins if $\varphi$ ever becomes true; otherwise, B wins.  The problem is to decide whether Player A has a winning strategy.

Our reduction from \textsc{ABF} to \textsc{LC\&Rp} is somewhat technical.  For this reason, we begin by presenting the construction and attempting to explain the intuition behind the reduction.  Only after that do we provide the formal proof.

\begin{subsection}{Construction}\label{sec:construction}

In reducing \textsc{ABF} to \textsc{LC\&Rp}, we must convert an instance $(X,Y,\varphi)$ of \textsc{ABF} into an instance $(G,\ell)$ of \textsc{LC\&Rp} such that Player A can win the ABF game if and only if $\ell$ cops can capture a robber on the graph $G$ (with specified protected edges).  The main difficulty, of course, is constructing $G$.  In this subsection we present the construction in small pieces.  Our aim is to provide intuition behind the construction, and as such we make several unjustified assertions, postponing formal proof until Section~\ref{sec:proof}.\\

In the construction below, to simplify the presentation, we incorporate a special vertex $h$.  All edges incident to $h$ are protected.  Thus, should the robber ever reach $h$, he may stay there indefinitely, thereby winning the game.  The concept of a ``winning vertex'' for the robber appeared both in the work of Goldstein and Reingold~\cite{GR95} (where such vertices were called ``holes'') and in that of Mamino~\cite{Mam13} (where they were called ``safe havens'').  Such a contrivance presents a fundamental difficulty: what is to stop the robber from simply starting the game at $h$?  Goldstein and Reingold dealt with this issue by explicitly specifying the players' initial positions, thereby specializing the game.  We adopt an approach more similar to Mamino's, wherein we later replace $h$ with a gadget that allows the robber to force the game back into a canonical ``initial'' configuration (as with the ``reset clique'' in the proof of Theorem~\ref{thm:lcrdp}).  

Every edge incident to $h$ is assumed to be protected, including the loop at $h$.  All other edges in $G$ are unprotected except where otherwise specified.  In the figures below, black vertices are unprotected, white vertices are protected, solid edges are unprotected, and dashed edges are protected.  A slashed circle denotes vertex $h$.\\

Let the given \textsc{ABF} instance have variables $X = \{x_1, x_2, \ldots, x_m\}$ and $Y = \{y_1, y_2, \ldots, y_n\}$ and CNF formula $\varphi$.  We aim to construct an ``equivalent'' instance of the \textsc{LC\&Rp} game in which the cops play the role of Player A, while the robber plays the role of Player B.  We encode the values of variables in both $X$ and $Y$ through the positions of cops in $G$.  For $i \in [m]$, the graph $G$ has vertices $x_i^T$ and $x_i^F$; we will ensure that at all times exactly one of these vertices is occupied by a cop.  A cop on $x_i^T$ indicates that variable $x_i$ is currently true, while a cop on $x_i^F$ indicates that $x_i$ is currently false.  Likewise, for each $j \in [n]$, the graph $G$ has vertices $y_j^T$ and $y_j^F$, which encode the truth value of $y_j$.  We refer to the vertices $x_i^T$, $x_i^F$, $y_j^T$, and $y_j^F$ as {\em variable vertices}.  We call the cops residing on these vertices {\em variable cops}, and say that the cop encoding the value of some variable is {\em assigned to} that variable.

We refer to the negation of some variable, through the appropriate movement of the corresponding cop, as a {\em shift} in that variable.  Our construction must provide means for the cops to force shifts in the $x_i$ and for the robber to force shifts in the $y_j$.  The former is straightforward: for each $i \in [m]$ we add an edge joining $x_i^T$ and $x_i^F$, and the cops may shift $x_i$ simply by moving the appropriate cop from one of these vertices to the other. 

Empowering the robber to force a shift in some $y_j$ is less straightforward.  We do this by employing the gadget shown on the left in Figure~\ref{fig:abf_gadgets}, adding one copy of the gadget for each $y_j$.  The vertices $v$ and $h$ are common to all gadgets, while all other vertices shown are unique to each copy of the gadget; $v$ is in some sense the robber's ``home base'' throughout the game.  Although the gadget may look complicated, the underlying idea is simple: should the robber (sitting on $v$) want to move a cop from, say, $y_j^T$ to $y_j^F$, he does so by himself moving to $a_j^F$.  This initiates a sequence of moves during which the cop must repeatedly defend against the robber's threats to reach $h$, while the robber must repeatedly evade capture by the cops.  In this way both players' moves are ``forced''; ultimately the cop must reach $y_j^F$, while the robber must return to $v$.  (The cops subsequently have the opportunity to shift some $x_i$, should they so desire.)  We refer to vertices of the form $a_j^T, a^{\prime T}_j, a_j^F$, and $a^{\prime F}_j$ as {\em primary robber vertices}, and to $b_j^T, b_j^{\prime T}, b_j^F,$ and $b_j^{\prime F}$ as {\em secondary robber vertices}.  As previously mentioned, we call $y_j^T$ and $y_j^F$ {\em variable vertices}; in addition, we call the $y_j^*$ {\em intermediate vertices}.\\

\begin{figure}[h]
\begin{center}
\includegraphics{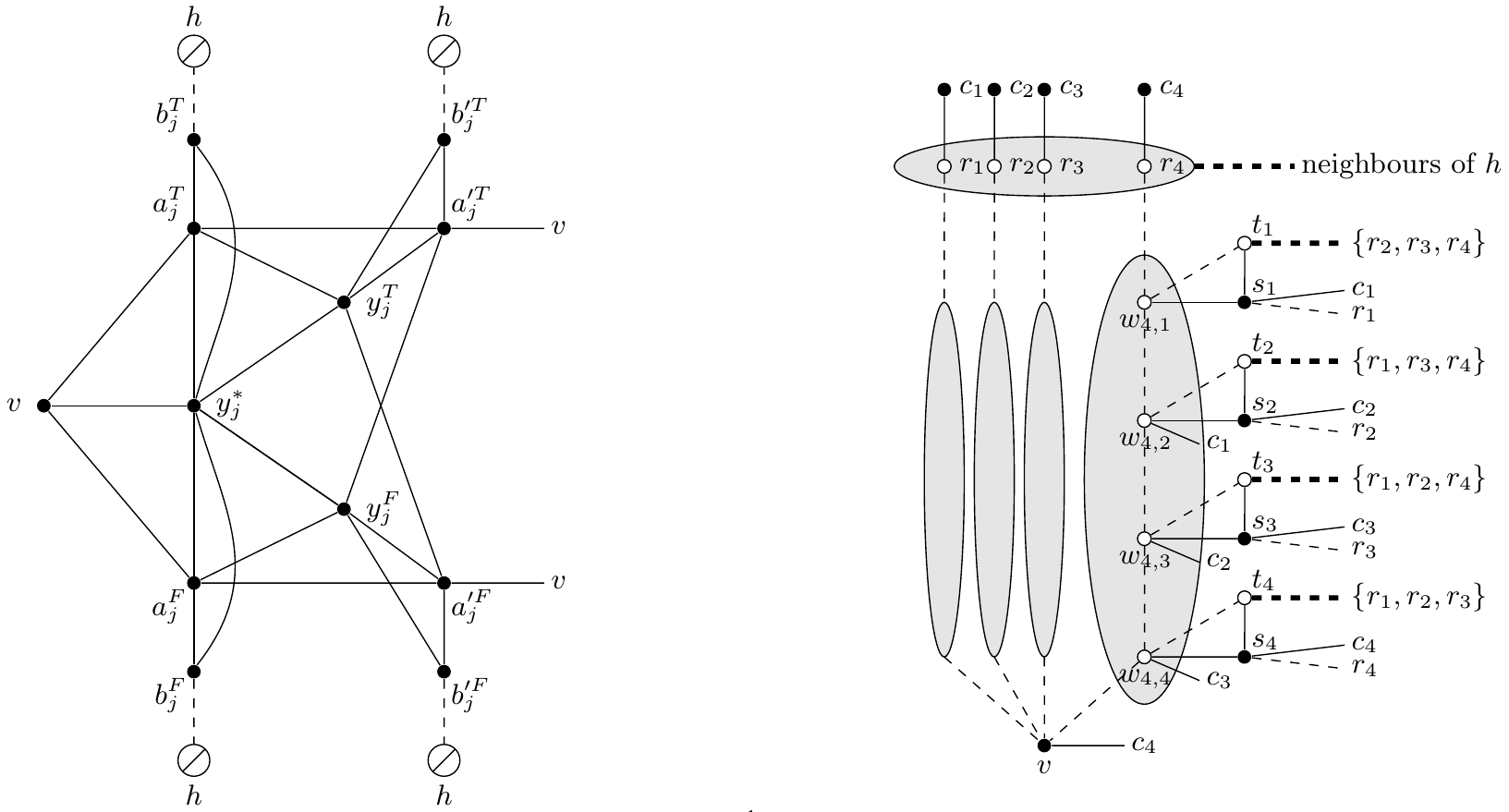}
\end{center}
\caption{. Left: $y_j$ gadget.  Right: reset gadget (with $k=4$). }\label{fig:abf_gadgets}
\end{figure}

\pagebreak

To restrict the variable cops to their assigned variable vertices, we add protected {\em restriction vertices} $u_z$ for all variables $z$.  We also add edges joining all restriction vertices to $h$ and to $v$; the edges incident to $h$ are protected.  Moreover, we add edges joining each $u_z$ to the variable vertices corresponding to $z$.  So long as all variable cops remain on their assigned variable vertices, they defend all restriction vertices.  However, should any variable cop ever leave her assigned variable vertices, the corresponding restriction vertex becomes undefended.  If at this time the robber occupies $v$ or some primary robber vertex, then he can safely move to the undefended restriction vertex and subsequently to $h$.\\  

The graph $G$ must also provide a means for the cops to capture the robber, should $\varphi$ ever become true.  To this end, we add one {\em clause vertex} $u_C$ for each clause $C$ of $\varphi$ and two additional vertices $u$ and $s$.   Vertex $u$ and all of the $u_C$ are adjacent to $h$, $s$, and $v$; all such edges incident to $h$ are protected.  For each CNF clause $C$, the vertex $u_C$ is made adjacent, in the usual way, to those variable vertices representing literals in $C$.  That is, if $C$ contains the literal $\alpha$, then we add edge $u_C\alpha^T$, and if $C$ contains the literal ${\overline \alpha}$, then we add edge $u_C\alpha^F$.  Finally, we add edges from $u$ to all robber vertices and restriction vertices, as well as edges from all clause vertices to all primary robber vertices.  

Suppose $\varphi$ becomes true after Player B's turn.  In the Cops and Robbers game, the robber sits on $v$ and it is the cops' turn.  We will ensure that, throughout the bulk of the game, a cop sits on $s$; to initiate the capture, this cop moves to $u$.  This forces the robber to flee to a clause vertex.  Should some clause of $\varphi$ be false, the robber can move to an undefended clause vertex, and subsequently to $h$.  However, should $\varphi$ be true, the robber has nowhere to run, and will be captured.  If $\varphi$ becomes true after Player A's turn, then the process works as before, except that the robber starts from some primary robber vertex.\\

We next replace the special vertex $h$ with a more complex gadget, as promised earlier.  Let $k=m+n+2$.  In place of $h$ we add unprotected vertices $c_1, c_2, \ldots, c_{k}$ and protected vertices $r_1, r_2, \ldots, r_k$; we refer to $\{r_1, r_2, \ldots, r_k\}$ as the {\em reset clique}.  We add unprotected edges $c_ir_i$ for all $i \in [k]$ and protected edges joining all pairs of vertices in the reset clique.  Finally, we replace each edge originally incident to $h$ with $k$ edges, incident to each of $r_1, r_2, \ldots, r_k$.  As all edges incident to $h$ were protected, so too are the new edges that replace them.  The key idea is that if the robber safely enters the reset clique, then the cops can only force him out by occupying $c_1, c_2, \ldots, c_k$.  We next provide means for the robber to force the cops from these vertices into particular ``starter'' vertices.

Toward this end, we add new vertices $s_1, s_2, \ldots, s_k$, $t_1, t_2, \ldots, t_k$, and $w_{i,j}$ for $i,j \in [k]$; of these, all but the $s_i$ are protected.  For all $i \in [k]$ we add edge $c_is_i$ and protected edge $r_is_i$.  We also add edges $s_it_i$ for all $i \in [k]$ and protected edges $t_ir_j$ for all $i,j \in [k]$ with $i \not = j$.  For all $i,j \in [k]$ we add edge $w_{i,j}s_j$.  We also add protected edges $w_{i,j}w_{i,j'}$ for all $i,j,j' \in [k]$, as well as protected edges $r_iw_{i,1}$ and $w_{i,k}v$ for all $i \in [k]$, except that we replace $r_1w_{1,1}$ with $r_1w_{1,2}$.  Additionally, for $i \in [k-1]$, we add edges from $c_i$ to all $w_{j,i+1}$, except that we replace $c_1w_{1,2}$ and $c_2w_{1,3}$ with $c_1w_{1,3}$ and $c_2w_{1,1}$, respectively.  Finally, we add edges joining $c_k$ to all clause vertices.  We strongly urge the reader to refer to Figure~\ref{fig:abf_gadgets}. 

Intuitively, this gadget allows either player to force a specific configuration of the game in which the robber occupies $v$ while the cops occupy $s_1, s_2, \ldots, s_k$.  We want this special configuration to be a starting point for our simulation of the \textsc{ABF} game.  Thus for $i \in [m]$, we identify $s_i$ with $x_i^T$ if $x_i$ is initially true and with $x_i^F$ otherwise.  Similarly, for $i \in [n]$ we identify $s_{m+i}$ with $y_i^T$ if $y_i$ is initially true and with $y_i^F$ otherwise.  Finally, we identify $s_{k-1}$ with $s$ and $c_{k-1}$ with $u$.

For technical reasons, it will be useful to ``forbid'' the robber from entering the variable and intermediate vertices.  To facilitate this, we add an unprotected vertex $c^*$ and a protected vertex $r^*$.  We add edges joining $c^*$ to $r^*$ and to all variable and intermediate vertices.  We also add protected edges joining $r^*$ to every other vertex in the graph (aside from $c^*$).  Clearly the cops should prevent the robber from ever reaching $r^*$.  Doing so requires having a cop occupy $c^*$ at all times; this has the side-effect of ensuring that the variable and intermediate vertices are always defended.\\

This completes the construction.  As we argue in Section~\ref{sec:proof}, the structure of $G$ greatly restricts the players' options.  Let the {\em initial configuration} be the configuration of the game in which the cops occupy $s_1, s_2, \ldots, s_k$ and $c^*$, the robber occupies $v$, and it is the cops' turn; note that the initial configuration corresponds to the starting configuration of the given \textsc{ABF} instance, in the sense that the variable cops encode their assigned vertices' given initial values.  Let a {\em basic configuration} be one in which the robber occupies $v$, one cop occupies $c^*$, one cop occupies $s_k$, one cop occupies $s$, and for each variable $z$, the variable cop assigned to $z$ occupies $z^T$ if $z$ is currently true in the \textsc{ABF} game and $z^F$ otherwise.  The reset gadget forces the game to quickly reach the initial configuration; our intuition, which we formalize in Section~\ref{sec:proof}, is that the game ``should'' regularly return to a basic configuration, provided the players play ``properly''.\\

With the game in a basic configuration, if it is the cops' turn, then the cops may force a shift in some $x_i$, which corresponds to Player A changing the value of $x_i$ in the \textsc{ABF} game.  Likewise, if it is the robber's turn, then the robber may force a shift in some $y_j$, which corresponds to Player B changing the value of $y_j$.  Both players also have the option of sitting still, which corresponds to a player changing no variables on his turn of the \textsc{ABF} game.  In any case the game returns to a basic configuration.  The game continues in this manner until the cop on $s$ moves to $u$, which results either in the capture of the robber (if $\varphi$ is true) or the escape of the robber (if $\varphi$ is false).  Thus the cops have a winning strategy on $G$ if and only if Player A has a winning strategy in the \textsc{ABF} game.  We remark that $G$ may clearly be constructed in polynomial time.

We conclude this subsection with a table summarizing important adjacencies in the construction.  (For the sake of brevity, we have omitted some vertices and edges appearing in the reset gadget; refer to Figure~\ref{fig:abf_gadgets}.)\\

{\tiny
\def\arraystretch{1.2}
\begin{table}[ht]
\centering
\begin{tabular}{|l|l|l|l|l|}
\hline

{\bf Vertices} &{\bf Unprot. edges to...} &{\bf Prot. edges to...} &{\bf Remark}\\

\hline\hline

$h$ & &$r^*$, $u$ &symbolizes the reset clique\\
 & &clause vertices &protected\\
 & &restriction vertices &\\

\hline

$v$ &$u, c_k$ &$r^*$  &\\
 &restriction vertices & &\\
 &clause vertices & &\\
 &[see also Figure~\ref{fig:abf_gadgets}] & &\\
 
\hline

$u$ &$s, v$ &$r^*$, $h$ &also known as $c_{k-1}$\\
 &robber vertices & &\\
 &restriction vertices & &\\

\hline

$s$ &$u$ &$r^*$ &also known as $s_{k-1}$\\
 &clause vertices & &\\

\hline 

$x_i^T, x_i^F$ &$c^*$, $u_{x_i}$ &$r^*$ &variable vertices\\
 &each other & &\\
 &appropriate clause vertices & &\\
 
\hline
 
$y_j^T, y_j^F$ &$c^*$, $u_{y_j}$ &$r^*$ &variable vertices\\
 &appropriate clause vertices & &\\
 &[see also Figure~\ref{fig:abf_gadgets}] & &\\
 
\hline

$y_j^*$ &$c^*$ &$r^*$ &intermediate vertices\\
 &[see also Figure~\ref{fig:abf_gadgets}] & &\\
 
\hline

$a_j^T, a_j^F, a^{\prime T}_j, a^{\prime F}_j$ &$u$ &$r^*$ &primary robber vertices\\
 &clause vertices & &\\
 &[see also Figure~\ref{fig:abf_gadgets}] & &\\
 
\hline

$b_j^T, b_j^F, b_j^{\prime T}, b_j^{\prime F}$ &$u$ &$r^*$ &secondary robber vertices\\
 &[see also Figure~\ref{fig:abf_gadgets}] & &\\
 
\hline

$u_z$ &$v, u$ &$r^*$, $h$ &restriction vertices\\
 &corresponding variable vertices & &\\
 
\hline

$u_C$ &$s, v, c_k$ &$r^*$, $h$ &clause vertices\\
 &corresponding variable vertices & &protected\\
 &primary robber vertices & &\\

\hline
\end{tabular}
\label{tab:abf_construction}
\caption{. Summary of the construction in Section~\ref{sec:construction}.}
\end{table}
}
\end{subsection}

\begin{subsection}{Reduction}\label{sec:proof}

We are now ready to prove that the \textsc{LC\&Rp} game on $G$ behaves as claimed in Section~\ref{sec:construction}.  In particular, given an ABF instance with $\ell$ variables, we show that Player A can win the ABF game if and only if $\ell+3$ cops can capture a robber on $G$.  In the argument below, we must carefully consider all possible moves by the cops and robber.  As in Section~\ref{sec:prelim}, we say that a player {\em must not} make a move if that move clearly provides no benefit, and we assume neither player makes any such moves.  The key to the reduction is that very few potential moves remain.

\pagebreak

\begin{theorem}\label{thm:ABF_proof}
\textsc{ABF} reduces to \textsc{LC\&Rp} in polynomial time.
\end{theorem}
\begin{proof}
Given an \textsc{ABF} instance $(X,Y,\varphi)$, we construct the graph $G$ as outlined in Section~\ref{sec:construction}.  Denote the sizes of $X$ and $Y$ by $m$ and $n$, respectively, and let $k=m+n+2$.  We play the \textsc{ABF} game and the \textsc{LC\&Rp} game (with $k+1$ cops) simultaneously, with each player using a strategy for the \textsc{ABF} game to guide his or her play in the \textsc{LC\&Rp} game.  Before we give explicit strategies for the players, we prove several claims about the flow of the game.  Our aim is to show that the players do not actually have much freedom in choosing their moves.  Recall that we suppose $\varphi$ is in conjunctive normal form; that is, $\varphi = C_1 \wedge C_2 \wedge \cdots \wedge C_{\ell}$, where each $C_i$ is the disjunction of literals.  We assume $\varphi$ is initially false, since otherwise Player A trivially wins the \textsc{ABF} game.  

\medskip\medskip

\noindent {\bf Claim 1a}: If the robber stands in the reset clique and it is his turn, then he can either escape capture forever or force the game to reach an initial configuration.\\

We assume throughout that some cop always occupies $c^*$.  If not, the robber can move to $r^*$ and remain there until some cop moves to $c^*$, at which point the robber can move anywhere in $G$; this clearly cannot benefit the cops. 

The cops can defend the entire reset clique only by occupying vertices $c_1, c_2, \ldots, c_k$.  Until this happens, the robber can always move (or remain on) to some undefended vertex in the clique. 

Suppose now that the cops occupy $c_1, c_2, \ldots, c_k$, the robber occupies some $r_{\ell}$, and it is the robber's turn.  We suppose $\ell \not = 1$; the other case is similar.  The robber moves to $w_{\ell,1}$.  Only $s_1$ defends $w_{\ell,1}$, so the robber cannot be captured unless some cop first moves to $s_1$. 

If some cop reaches $s_1$ while the remaining cops still occupy $c_2, c_3, \ldots, c_k$, then the robber moves to $w_{\ell,2}$.  Suppose instead that for some $j \ge 2$, the cop on $c_j$ moves.  The robber now moves to $t_1$.  At this point the cops occupy all $c_i$ for $i \not \in \{1,j\}$, and no cop occupies $s_1$.  If some cop (other than the cop on $c^*$) moves to $c_j$, then the robber returns to $w_{\ell,1}$ and proceeds as before.  Otherwise, vertex $r_j$ remains undefended; the robber moves there, gaining access to the reset clique, and the process begins anew.

Hence we may assume that the robber occupies $w_{\ell,2}$, that cops occupy $s_1, c_2, c_3, \ldots, c_k$, and that it is the cops' turn.  If the cop on $s_1$ leaves, then the robber returns to $w_{\ell,1}$ and plays as in the preceding paragraph.  If the cop on some $c_j$ moves for $j \ge 3$, then the robber moves to $t_2$; at this point both $r_1$ and $r_j$ are undefended and the cops cannot defend both with their next move, so the robber can safely return to the reset clique.  If the cop on $c_2$ moves anywhere other than $s_2$, then the robber moves to $s_2$; from here, if the cop returns to $c_2$ then the robber returns to $w_{\ell,2}$, and otherwise the robber can safely move to $r_2$.  If the cops all remain in place, then so does the robber.  None of the preceding cases benefit the cops; the only remaining possibility is that the cop on $c_2$ moves to $s_2$, to which the robber responds by moving to $w_{\ell,3}$.  

We now assume that the robber occupies $w_{\ell,3}$, that cops occupy $s_1, s_2, c_3, \ldots, c_k$, and that it is the cops' turn.  If the cops make any move that leaves $t_3$, $r_1$, and $r_2$ undefended, then the robber moves to $t_3$; the cops cannot defend both $r_1$ and $r_2$ with a single move, so no matter how they respond, the robber regains access to the reset clique.  If instead the cop on $s_2$ moves to $c_2$, then the robber returns to $w_{\ell,2}$ and plays as in the preceding paragraph.  Finally, if the cop on $s_1$ moves to $c_1$, then the robber moves to $w_{\ell,1}$ and uses a strategy symmetric to the one in the preceding paragraph.  In particular, if the cop on $s_2$ leaves, then the robber returns to $w_{\ell,1}$; if the cop on some $c_j$ moves for $j \ge 3$, then the robber moves to $t_2$ and subsequently regains access to the reset clique; if the cop on $c_1$ moves anywhere other than $s_1$, then the robber moves to $s_1$ and, from there, either returns to $w_{\ell,1}$ or moves to $r_1$ as appropriate; if all cops remain in place, so does the robber; finally, if the cop on $c_1$ moves to $s_1$, then the robber returns to $w_{\ell,3}$.  None of the preceding options benefit the cops, so we may assume that the cop on $c_3$ moves to $s_3$.  The robber responds by moving to $w_{\ell,4}$.

With the robber on $w_{\ell,4}$ the arguments become simpler.  If the cop on $c_4$ moves to $s_4$, then the robber moves to $w_{\ell,5}$; if the cops make any other move, then the robber moves to $t_4$ and subsequently to some undefended $r_i$.  Continuing in this way, the cops must ``push'' the robber down through the $w_{\ell,j}$ to $w_{\ell,k}$ and from there to $v$, at which point we reach the initial configuration.

\medskip\medskip

\noindent {\bf Claim 1b}: If the cops occupy vertices $c_1, c_2, \ldots, c_{k}$ and $c^*$, and it is the cops' turn, then the cops can either capture the robber or force the game to reach the initial configuration.\\

The cops defend all vertices aside from $w_{1,2}$, the $w_{\ell,1}$ (for $\ell \not = 1$), and the $t_i$, so we may assume the robber occupies one of these.  If the robber occupies some $t_i$, then the cop on $c_i$ moves to $s_i$, which ensures the robber's capture.  Hence we suppose the robber occupies either $w_{1,2}$ or $w_{\ell,1}$ for some $\ell \not = 1$.  The cases are similar; for convenience we consider only the latter case.  

At this point the cops occupy $c_1, c_2, \ldots, c_k$ and $c^*$, the robber occupies $w_{\ell,1}$, and it is the cops' turn.  The cop on $c_1$ moves to $s_1$.  The cops now defend $r_{\ell}$, $w_{\ell,1}$, $s_1$, and $t_1$, so the robber must move to $w_{\ell,2}$.  The cop on $c_2$ now moves to $s_2$, which forces the robber to $w_{\ell,3}$, and so forth.  Once the robber reaches $w_{\ell,k}$, the cop on $c_k$ moves to $s_k$; the robber must now move to $v$, which produces the initial configuration.

\medskip\medskip

\noindent {\bf Claim 2}: If the game is in a basic configuration and it is the robber's turn, then the robber must either remain in place, move to some vertex $a_j^T$ such that $y_j$ is currently false, or move to some vertex $a_j^F$ such that $y_j$ is currently true.\\

Vertex $v$ is adjacent to the restriction vertices, the clause vertices, the intermediate vertices, the primary robber vertices, $u$, $c_k$, the $w_{i,k}$, and $r^*$.  In a basic configuration, the cop on $c^*$ defends $r^*$, the cop on $s_k$ defends $c_k$ and the $w_{i,k}$, the cop on $s$ defends the clause vertices and $u$, and the variable cops defend the restriction vertices, the intermediate vertices, the $a^{\prime T}_j$, and the $a^{\prime F}_j$.  Moreover, the variable cop assigned to $y_j$ defends $a_j^T$ when $y_j$ is true and $a_j^F$ when $y_j$ is false.  Thus the only undefended vertices in $N[v]$ are $v$ itself, vertices $a_j^T$ for which $y_j$ is false, and vertices $a_j^F$ for which $y_j$ is true.

\medskip\medskip
 
\noindent {\bf Claim 3}: If the game is in a basic configuration (on the robber's turn), $\varphi$ is currently false, and the robber moves to some vertex $a_j^T$ such that $y_j$ is currently false (or to some $a_j^F$ such that $y_j$ is currently true), then $y_j$ shifts, after which the game returns to a basic configuration (on the cops' turn).\\

It follows from Claims 1a and 1b that the cops must not allow the robber safe access to the reset clique.  If the robber reaches the reset clique then he can force the game into the initial configuration, which clearly does not benefit the cops, as they already had an opportunity to force the initial configuration (as in Claim 1b).  Likewise, the cops must not allow the robber access to vertex $h$ (which in reality represents the reset clique).

By symmetry we assume that the robber moves to some $a_j^T$ such that $y_j$ is currently false.  The cops must now defend $b_j^T$, lest the robber move there and next to $h$.  Only four vertices defend $b_j^T$, namely $a_j^T$, $b_j^T$, $y_j^*$, and $u$.  No cops are adjacent to $a_j^T$ or $b_j^T$.  The only cop adjacent to $u$ is the cop on $s$.  Since $\varphi$ is false, it contains some false clause $C$; should the cop on $s$ move to $u$ she would leave $u_C$ undefended, so the robber could move there and subsequently to $h$.  The cop on $c^*$ cannot move to $y_j^*$, since this would leave $r^*$ undefended.  The only remaining option is for the cop on $y_j^F$ to move to $y_j^*$. 

The robber cannot remain on $a_j^T$, nor can he move to $b_j^T$, $y_j^T$, $y_j^*$, or $v$, lest the cop on $y_j^*$ capture him.  He also cannot move to $u$ or to any of the clause vertices, lest the cop on $s$ would capture him.  Thus he must move to $a^{\prime T}_j$.  As before, the cops must now defend $b_j^{\prime T}$, and as before, they must do so by moving the cop on $y_j^*$ to $y_j^T$.  Now the robber cannot remain on $a^{\prime T}_j$, the cop on $c^*$ defends $y_j^T$ and $y_j^F$, the cop on $s$ defends $u$ and the clause vertices, and the cop on $y_j^T$ defends $b_j^{\prime T}$ and $a_j^T$.  Hence the robber must move to $v$.  Thus $y_j$ has shifted, the game has returned to a basic configuration, and it is the cops' turn.

\medskip\medskip

\noindent {\bf Claim 4}: If the robber occupies $v$ or a primary robber vertex, all variable cops occupy their assigned variable vertices, some cop occupies $c^*$, some cop occupies $s_k$, some cop occupies $s$,  and $\varphi$ is currently false, then the cop on $s$ must not move.\\

Should the cop on $s$ move anywhere other than $u$ or $r_{k-1}$ she would leave $u$ undefended, allowing the robber to move to $u$ and from there to $h$.  Now suppose the cop on $s$ moves to $u$ or $r_{k-1}$.  Since $\varphi$ is false, some clause $C$ of $\varphi$ is false.  The vertex $u_C$ is defended only by $s$, $v$, $c_k$, the primary robber vertices, and those variable vertices corresponding to literals in $C$.  Since $C$ is false, $u_C$ is not defended by any variable cop, nor is it defended by the cop on $c^*$, the cop on $s_k$, or the final cop (who occupies either $u$ or $r_{k-1}$).  Thus the robber can safely move to $u_C$ and subsequently to $h$.

\medskip\medskip

\noindent {\bf Claim 5}: If the game is in a basic configuration, it is the cops' turn, and $\varphi$ is currently false, then either all cops remain in place, or the cops execute a shift in some $x_i$.\\

As noted in Claim 3, the cops must not allow the robber safe access to the reset clique (or to the symbolic vertex $h$).  If the cop on $c^*$ moves anywhere, then the robber moves to $r^*$ and subsequently to some undefended vertex in the reset clique.  If the cop $C$ on $s_k$ moves to $t_k$ or to some $w_{i,k}$, then the robber moves to $c_k$ and, on his next turn, to $r_k$ (note that the cops cannot defend $r_k$ before the robber reaches it).  If instead $C$ moves to $c_k$ or $r_k$, then the robber moves to $w_{1,k}$; if $C$ next returns to $s_k$ then the robber returns to $v$ (an exchange that does not benefit the cops), and otherwise the robber moves to $t_k$ and subsequently to some undefended vertex of the reset clique.  By Claim 4, the cop on $s$ must not move.  Only the variable cops remain.  The variable cops must not leave their assigned variable vertices; this would leave some restriction vertex undefended, allowing the robber to move there and next to $h$.  Hence the only variable cops that may move are those assigned to the $x_i$, and their only option is to execute a shift.  

\medskip\medskip

\noindent {\bf Claim 6}: If any shift makes $\varphi$ true, then the cops can capture the robber.\\

By the preceding claims the game eventually reaches a basic configuration, after which shifts are executed in sequence until $\varphi$ becomes true.  Moreover, a shift in some $x_i$ yields a basic configuration in which it is the robber's turn, while a shift in some $y_j$ yields a basic configuration in which it is the cops' turn.

Suppose first that $\varphi$ becomes true after a shift in some $y_j$.  From the resulting basic configuration, the cop on $s$ may move to $u$, threatening to capture the robber.  The robber clearly cannot remain on $v$.  Vertex $v$ is adjacent to $u$, $c_k$, the restriction vertices, the clause vertices, the intermediate vertices, the primary robber vertices, and the $w_{i,k}$.  The variable cops defend the restriction vertices and intermediate vertices, the cop on $s_k$ defends $c_k$ and the $w_{i,k}$, and the cop on $u$ defends the primary robber vertices and $u$ itself.  Thus the robber must move to some clause vertex.  However, each clause of $\varphi$ has at least one true literal, and hence each clause vertex is defended by at least one variable cop.  No matter where the robber moves, the cops can capture him on their next turn.

Now suppose $\varphi$ becomes true after a shift in some $x_i$.  By Claim 3, the robber must either stay on $v$ or move to some $a_j^T$ or $a_j^F$.  Regardless of the robber's choice, the cop on $s$ may move to $u$, threatening to capture the robber.  As above, $v$ and all its neighbours are defended; additionally, the cop on $u$ defends all secondary robber vertices, and the cop on $c^*$ defends all variable vertices.  No matter which vertex the robber currently sits on, he cannot escape capture.

\medskip\medskip

We are finally ready to give the players' strategies.  Suppose first that Player A has a winning strategy for the \textsc{ABF} game; we show how $k+1$ cops can capture the robber on $G$.  The cops initially occupy $c_1, c_2, \ldots, c_{k}$ and $c^*$; from there, they force the game into the initial configuration (which is also a basic configuration) as in Claim 1b.

From a basic configuration, on the cops' turns, they play as follows.  If $\varphi$ is true, then the cops capture the robber as in Claim 6.  Otherwise, the cops consult Player A's strategy in the \textsc{ABF} game.  If Player A changes the value of some $x_i$, then the cops execute a shift on $x_i$.  If instead Player A changes no variables, then the cops all remain in place.  In either case the \textsc{LC\&Rp} game remains in a basic configuration, but it is now the robber's turn.

From a basic configuration, on the robber's turns, he has very few options.  By Claims 2 and 3, the robber must either remain in place or initiate a shift on some $y_j$.  In the former case, the cops act as if Player B changed no variables in the \textsc{ABF} game, while in the latter case, they act as if he changed the value of $y_j$.  In either case, the \textsc{LC\&Rp} game soon returns to a basic configuration by Claim 3.

Play continues in this manner.  Since the cops follow a winning strategy for Play A in the \textsc{ABF} game, eventually some shift makes $\varphi$ true.  By Claim 6, the cops can then capture the robber.\\

Now suppose Player B has a winning strategy for the \textsc{ABF} game; we show how the robber can perpetually evade $k$ cops on $G$.  Initially, if $r^*$ or any vertex of the reset clique is undefended, then the robber starts there; otherwise the cops must occupy $c_1, c_2, \ldots, c_k$ and $c^*$, and the robber starts on $w_{2,1}$.  The robber forces the game into the initial configuration as in Claim 1a and, as before, the cops and robber subsequently execute a series of shifts.  When the cops shift $x_i$, the robber acts as if Player A changed the value of $x_i$.  From a basic configuration on the robber's turns, he consults Player B's strategy in the \textsc{ABF} game.  If Player B changes the value of some $y_j$, then the robber initiates a shift on $y_j$, and the game soon returns to a basic configuration by Claim 3.  By Claims 2-5, the \textsc{LC\&Rp} game continues in this manner so long as $\varphi$ is false; since Player B perpetually prevents $\varphi$ from becoming true in the \textsc{ABF} game, the robber perpetually escapes capture in the \textsc{LC\&Rp} game.\\

As noted in Section~\ref{sec:construction}, the \textsc{LC\&Rp} instance can clearly be constructed in time polynomial in the size of the \textsc{ABF} input.    
\end{proof}

\medskip\medskip

Thus \textsc{ABF} reduces to \textsc{LC\&Rp} (Theorem~\ref{thm:ABF_proof}), \textsc{LC\&Rp} reduces to \textsc{C\&Rp} (Theorem~\ref{thm:lcrdp}), and \textsc{C\&Rp} reduces to \textsc{C\&R} (\cite{Mam13}, Lemma 3.1).  Consequently, since \textsc{ABF} is \class{EXPTIME}-hard (\cite{SC79}), so is \textsc{C\&R}.  As noted in the introduction, \textsc{C\&R} clearly belongs to \class{EXPTIME}, so Conjecture~\ref{conj:main} follows.  
\end{subsection}

\section*{Acknowledgements}

The author thanks Anthony Bonato for his careful reading of and many useful comments on the paper.  The author is also greatly indebted to Paul Hunter for helping resolve a subtle but significant issue in the proof of the main result.

\end{section}
\end{document}